\documentclass[a4paper,10pt,leqno]{amsart}
\usepackage{amsmath,amsfonts,amsthm,amssymb,dsfont}
\usepackage[alphabetic]{amsrefs}
\usepackage[OT4]{fontenc}
\usepackage{enumerate}
\usepackage{mathrsfs}
\usepackage{enumerate, xspace}
\usepackage[pdftex]{graphicx}
\usepackage{color}

\long\def\symbolfootnote[#1]#2{\begingroup
\def\thefootnote{\fnsymbol{footnote}}\footnote[#1]{#2}\endgroup}

\newtheorem{theorem}{Theorem}[section]
\newtheorem{lemma}[theorem]{Lemma}
\newtheorem{thm}[theorem]{Theorem}
\newtheorem{sublemma}[theorem]{Sublemma}
\newtheorem{prop}[theorem]{Proposition}

\theoremstyle{definition}
\newtheorem{rem}[theorem]{Remark}
\newtheorem{defin}[theorem]{Definition}

\newcommand{\executeiffilenewer}[3]{%
\ifnum\pdfstrcmp{\pdffilemoddate{#1}}%
{\pdffilemoddate{#2}}>0%
{\immediate\write18{#3}}\fi%
}

\begin{document}

\title{Slim unicorns and uniform hyperbolicity for arc graphs and curve graphs}

\author[S.~Hensel]{Sebastian Hensel}
           \address{The University of Chicago, Department of Mathematics\\
5734 South University Avenue, Chicago, Illinois 60637--1546 USA}
           \email{hensel@uchicago.edu}

\author[P.~Przytycki]{Piotr Przytycki$^{\dag}$}
\address{Inst. of Math., Polish Academy of Sciences\\
 \'Sniadeckich 8, 00-956 Warsaw, Poland}
\email{pprzytyc@mimuw.edu.pl}
\thanks{$\dag$ Partially supported by MNiSW grant N N201 541738 and the Foundation for Polish Science.}

\author[R.~C.~H.~Webb]{Richard C.~H.~Webb}
           \address{University of Warwick, Mathematics Institute\\
Zeeman Building, Coventry CV4 7AL -- UK}
           \email{r.c.h.webb@warwick.ac.uk}

\maketitle

\begin{abstract}
\noindent
We describe unicorn paths in the arc graph and show that they form $1$--slim triangles and are invariant under taking subpaths.
We deduce that all arc graphs are $7$--hyperbolic. Considering the same paths in the arc and curve
graph, this also shows that all curve graphs are $17$--hyperbolic, including closed surfaces.
\end{abstract}

\section{Introduction}
\label{sec:intro}
The \emph{curve graph $\mathcal{C}(S)$} of a compact oriented surface
$S$ is the graph whose vertex set is the set of homotopy classes of
essential simple closed curves and whose edges correspond to disjoint curves.
This graph has turned out to be a fruitful tool in the study of both
mapping class groups of surfaces and of hyperbolic $3$--manifolds.
One prominent feature is that $\mathcal{C}(S)$ is a \emph{Gromov hyperbolic} space (when one endows each edge with length $1$) as was
proven by Masur and Minsky \cite{MM}.
The main result of this paper is to give a new (short and self-contained) proof with a low uniform constant:
\begin{thm}
\label{thm:curve}
If $\mathcal{C}(S)$ is connected, then it is $17$--hyperbolic.
\end{thm}
Here, we say that a connected graph $\Gamma$ is \emph{$k$--hyperbolic}, if all
of its triangles formed by geodesic edge-paths are $k$--centred. A
triangle is \emph{$k$--centred at a vertex $c\in \Gamma^{(0)}$}, if
$c$ is at distance $\leq k$ from each of its three sides. This notion of
hyperbolicity is equivalent (up to a linear change in the constant) to
the usual slim-triangle condition \cite{Sh}.

After Masur and Minsky's original proof, several other proofs for the
hyperbolicity of $\mathcal{C}(S)$ were given.
Bowditch proved that $k$ can be chosen to grow logarithmically
with the complexity of $S$ \cite{Bow}. A different proof of hyperbolicity
was given by Hamenst\"adt \cite{Ham}. Recently, Aougab \cite{A},
Bowditch \cite{Bow2}, and Clay, Rafi and Schleimer \cite{CRS} have
proved, independently, that $k$ can be chosen independent of $S$.

Our proof of Theorem~\ref{thm:curve} is based on a careful study of Hatcher's
surgery paths in the arc graph $\mathcal{A}(S)$ \cite{Ha0}. The key point is that these
paths form $1$--slim triangles (Section~\ref{sec:surgery}), which
follows from viewing surgered arcs as \emph{unicorn arcs}
introduced as one-corner arcs in \cite{HOP}. We then use a hyperbolicity argument of Hamenst\"adt \cite{Ham}, which provides a better constant than a similar criterion due to Bowditch \cite[Prop~3.1]{Bow2}. This
gives rise to uniform hyperbolicity of the arc graph
(Section~\ref{sec:criterion}) and then also of the curve graph (Section~\ref{sec:curve}). Thus, we also prove:

\begin{thm}
\label{thm:arc}
$\mathcal{A}(S)$ is $7$--hyperbolic.
\end{thm}

The arc graph was proven to be hyperbolic by Masur and Schleimer
\cite{MS}, and recently another proof has been given by Hilion and
Horbez \cite{HH}. Uniform hyperbolicity, however, was not known.

\section{Preliminaries}

Let $S$ be a compact oriented topological surface. We consider arcs on $S$ that are properly embedded and \emph{essential}, i.e.\ not homotopic into
$\partial S$. We also consider embedded closed curves on $S$ that are not homotopic to a point or into $\partial S$.
The \emph{arc and curve graph} $\mathcal{AC}(S)$ is the graph
whose vertex set $\mathcal{AC}^{(0)}(S)$ is the set of homotopy classes of
arcs and curves on $(S,\partial S)$. Two vertices are connected by an edge in $\mathcal{AC}(S)$ if the
corresponding arcs or curves can be realised disjointly. The \emph{arc graph} $\mathcal{A}(S)$ is the subgraph of $\mathcal{AC}(S)$ induced on the vertices that are homotopy classes of arcs. Similarly,
the \emph{curve graph} $\mathcal{C}(S)$ is the subgraph of $\mathcal{AC}(S)$ induced on the vertices that are homotopy classes of curves.

Let $a$ and $b$ be two arcs on $S$. We say that $a$ and
$b$ are in \emph{minimal position} if the number of intersections between $a$ and $b$ is minimal in the
homotopy classes of $a$ and $b$. It is well known that this is equivalent to $a$ and $b$ being transverse and having no discs in $S-(a\cup b)$ bounded by a subarc of $a$ and a subarc of $b$ (\emph{bigons}) or bounded by a subarc of $a$, a subarc of $b$ and a subarc of $\partial S$ (\emph{half-bigons}).

\section{Unicorn paths}
\label{sec:surgery}

We now describe Hatcher's surgery paths \cite{Ha0} in the guise of unicorn paths.

\begin{defin}
Let $a$ and $b$ be in minimal position. Choose endpoints $\alpha$ of $a$ and $\beta$ of $b$. Let $a'\subset a, b'\subset b$ be subarcs with endpoints $\alpha,\beta$ and a common endpoint $\pi$ in $a\cap b$. Assume that $a'\cup b'$ is an embedded arc.
Since $a,b$ were in minimal position, the arc $a'\cup b'$ is essential. We say that $a'\cup b'$ is a
\emph{unicorn arc obtained from $a^\alpha,b^\beta$}. Note that it is uniquely determined by $\pi$, although not all $\pi\in a\cap b$ determine unicorn arcs, since the components of $a-\pi, b-\pi$ containing $\alpha,\beta$ might intersect outside $\pi$.

We linearly order unicorn arcs so that $a'\cup b'\leq a''\cup b''$ if and only if $a''\subset a'$ and $b'\subset b''$. Denote by $(c_1,\ldots, c_{n-1})$ the ordered set of unicorn arcs. The sequence $\mathcal{P}(a^\alpha,b^\beta)=(a=c_0,c_1,\ldots, c_{n-1}, c_n=b)$ is called the \emph{unicorn path between $a^\alpha$ and $b^\beta$}.
\end{defin}

The homotopy classes of $c_i$ do not depend on the choice of representatives of the homotopy classes of $a$ and $b$.

\begin{rem}
Consecutive arcs of the unicorn path represent adjacent vertices in the arc graph.
Indeed, suppose $c_i=a'\cup b'$ with $2\leq i\leq n-1$ and let $\pi'$ be the first point on $a-a'$ after $\pi$ that lies on $b'$.
Then $\pi'$ determines a unicorn arc. By definition of $\pi'$, this arc is $c_{i-1}$. Moreover, it can be homotoped off $c_i$, as desired.
The fact that $c_0c_1$ and $c_{n-1}c_n$ form edges follows similarly.
\end{rem}

We now show the key $1$--slim triangle lemma.

\begin{lemma}
\label{lem:triangle}
Suppose that we have arcs with endpoints $a^\alpha,b^\beta,d^\delta$, mutually in minimal position. Then for every $c\in \mathcal{P}(a^\alpha, b^\beta)$,
there is $c^*\in \mathcal{P}(a^\alpha,d^\delta)\cup \mathcal{P}(d^\delta,b^\beta)$, such that $c,c^*$ represent adjacent vertices in $\mathcal{A}(S)$.
\end{lemma}

\begin{proof}
If $c=a'\cup b'$ is disjoint from $d$, then there is nothing to prove. Otherwise, let $d'\subset d$ be the maximal subarc with endpoint $\delta$ and with interior disjoint from $c$. Let $\sigma\in c$ be the other endpoint of $d'$. One of the two subarcs into which $\sigma$ divides $c$ is contained
in $a'$ or $b'$. Without loss of generality, assume that it is contained in $a'$, denote it by $a''$. Then $c^*=a''\cup d'\in \mathcal{P}(a^\alpha,d^\delta)$.
Moreover, $c^*$ and $c$ represent adjacent vertices in $\mathcal{A}(S)$, as desired.
\end{proof}

Note that we did not care whether $c$ was in minimal position with $d$ or not. A slight enhancement shows that the triangles are $1$--centred:

\begin{lemma}
\label{lem:minsize1}
Suppose that we have arcs with endpoints $a^\alpha,b^\beta,d^\delta$, mutually in minimal position. Then there are pairwise adjacent vertices on $\mathcal{P}(a^\alpha, b^\beta),\mathcal{P}(a^\alpha,d^\delta)$ and $\mathcal{P}(d^\delta,b^\beta)$.
\end{lemma}

\begin{proof}
If two of $a,b,d$ are disjoint, then there is nothing to prove. Otherwise for unicorn arcs $c_i=a'\cup b', c_{i+1}=a''\cup b''$
let $\pi,\sigma$ their intersection points with $d$ closest to $\delta$ along $d$. There is $0\leq i <n$ such that $\pi\in a', \sigma\in b''$. Without loss of generality assume that $\pi$ is not farther than $\sigma$ from $\delta$. Let $\pi'$ be the intersection point of $a$ with the subarc $\delta\sigma\subset d$ that is closest to $\alpha$ along $a$. Then $c_{i+1}$, the unicorn arc obtained from $d^\delta,b^\beta$ determined by $\sigma$, and the unicorn arc obtained from $a^\alpha,d^\delta$ determined by $\pi'$, represent three adjacent vertices in $\mathcal{A}(S)$.
\end{proof}

We now prove that unicorn paths are invariant under taking subpaths, up to one exception.

\begin{lemma}
\label{lem:subpath}
For every $0\leq i<j\leq n$, either $\mathcal{P}(c^\alpha_i,c^\beta_j)$ is a subpath of $\mathcal{P}(a^\alpha,b^\beta)$, or $j=i+2$ and $c_i,c_j$
represent adjacent vertices of $\mathcal{A}(S)$.
\end{lemma}

Before we give the proof, we need the following.
\begin{sublemma}
\label{sub:minimal_position}
Let $c=c_{n-1}$, which means that $c=a'\cup b'$ with the interior of $a'$ disjoint from $b$. Let $\tilde{c}$ be the arc
homotopic to $c$ obtained by homotopying $a'$ slightly off $a$ so that $a'\cap\tilde{c}=\emptyset$.
Then either $\tilde{c}$ and $a$ are in minimal position, or they bound exactly one half-bigon, shown in Figure 1.
In that case, after homotopying $\tilde{c}$ through that half-bigon to $\bar{c}$, the arcs $\bar{c}$ and $a$ are already in minimal position.
\end{sublemma}

\begin{proof}
Let $\tilde{\alpha}$ be the endpoint of $\tilde{c}$ corresponding to $\alpha$ in $c$.
The arcs $\tilde{c}$ and $a$ cannot bound a bigon, since then $b$ and $a$ would bound a bigon contradicting minimal position.
Hence if $\tilde{c}$ and $a$ are not in minimal position, then they bound a half-bigon $\tilde{c}'a''$, where $\tilde{c}'\subset\tilde{c}, a''\subset a$.
Let $\pi'=\tilde{c}'\cap a''$. The subarc $\tilde{c}'$ contains $\tilde{\alpha}$, since otherwise $a$ and $b$ would bound a half-bigon.
Since the interior of $a'$ is disjoint from $b$, by minimal position of $a$ and $b$ the interior of $a''$ is also disjoint from $b$.
In particular, $a''$ does not contain $\alpha$, since otherwise $a'\subsetneq a''$ and $\pi$ would lie in the interior of $a''$.
Moreover, $\pi$ and $\pi'$ are consecutive intersection points with $a$ on $b$ (see Figure 1).

Let $b''$ be the component of $b-\pi'$ containing $\beta$. Let $\bar{c}$ be obtained from $a''\cup b''$ by homotopying it off $a''$.
Applying to $\bar{c}$ the same argument as to $\tilde{c}$, but with the endpoints of $a$ interchanged,
we get that either $\bar{c}$ is in minimal position with $a$ or there is a a half-bigon $\bar{c}'a'''$, where $\bar{c}'\subset \bar{c}, a'''\subset a$. But in the latter case we have $\alpha\in a'''$, which implies $a'\subsetneq a'''$ contradicting the fact that the interior of $a'''$ should be
disjoint from $b$.
\end{proof}

\begin{figure}
\executeiffilenewer{fig.svg}{fig.pdf}%
{inkscape -z -D --file=fig.svg %
--export-pdf=fig.pdf --export-latex}%
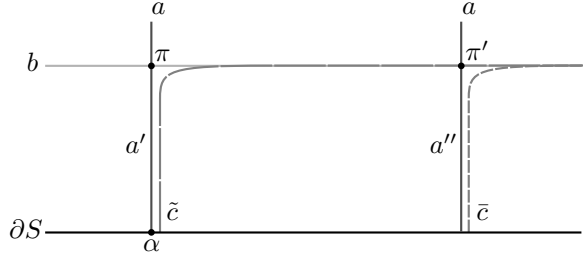%

\caption{The only possible half-bigon between $\tilde{c}$ and $a$}
\end{figure}

\begin{proof}[Proof of Lemma~\ref{lem:subpath}]
We can assume $i=0$, so that $c_i=a$, and $j=n-1$, so that $c_j=a'\cup b'$, where $a'$ intersects $b$ only at its endpoint $\pi$ distinct from $\alpha$. Let $\tilde{c}$ be obtained from $c=c_j$ as in Sublemma~\ref{sub:minimal_position}. If $\tilde{c}$ is in minimal position with $a$, then
points in $(a\cap b)-\pi$ determining unicorn arcs obtained from $a^\alpha,b^\beta$ determine the same unicorn arcs obtained from
$a^\alpha,\tilde{c}^\beta$, and exhaust them all, so we are done.

Otherwise, let $\bar{c}$ be the arc from Sublemma~\ref{sub:minimal_position} homotopic to $c$ and in minimal position with $a$.
The points $(a\cap b)-\pi-\pi'$ determining unicorn arcs obtained from $a^\alpha,b^\beta$ determine the same unicorn arcs obtained from
$a^\alpha,\bar{c}^\beta$. Let $a^*=a-a''$. If $\pi'$ does not determine a unicorn arc obtained from $a^\alpha,b^\beta$, i.e.\ if $a^*$ and $b''$
intersect outside $\pi'$, then we are done as in the previous case. Otherwise,
$a^*\cup b''=c_1$, since it is minimal in the order on the unicorn arcs obtained from $a^\alpha,b^\beta$.
Moreover, since the subarc $\pi\pi'$ of $a$ lies in $a^*$, its interior is disjoint from $b''$, hence also from $b'$.
Thus $a^*\cup b''$ precedes $c$ in the order on the unicorn arcs obtained from $a^\alpha,b^\beta$, which means that $j=2$, as desired.
\end{proof}

\section{Arc graphs are hyperbolic}
\label{sec:criterion}

\begin{defin}
To a pair of vertices $a,b$ of $\mathcal{A}(S)$ we assign the following family $P(a,b)$ of unicorn paths.
Slightly abusing the notation we realise them as arcs $a,b$ on $S$ in minimal position. If $a,b$ are disjoint, then let $P(a,b)$ consist of a single path $(a,b)$. Otherwise, let $\alpha_+,\alpha_-$ be the endpoints of $a$ and let $\beta_+,\beta_-$ be the endpoints of $b$.
Define $P(a,b)$ as the set of four unicorn paths:
$\mathcal{P}(a^{\alpha_+},b^{\beta_+}),\mathcal{P}(a^{\alpha_+},b^{\beta_-}),\mathcal{P}(a^{\alpha_-},b^{\beta_+}),$ and $\mathcal{P}(a^{\alpha_-},b^{\beta_-})$.
\end{defin}

The proof of the next proposition follows along the lines of \cite[Prop 3.5]{Ham} (or \cite[Thm III.H.1.7]{BH}).

\begin{prop}
\label{prop:close}
Let $\mathcal{G}$ be a geodesic in $\mathcal{A}(S)$ between vertices $a,b$. Then any vertex $c\in \mathcal{P}\in P(a,b)$ is at distance $\leq 6$ from $\mathcal{G}$.
\end{prop}

In the proof we need the following lemma which is immediately obtained by applying $k$ times Lemma~\ref{lem:triangle}.

\begin{lemma}
\label{lem:log}
Let $x_0,\ldots, x_m$ with $m\leq 2^k$ be a sequence of vertices in $\mathcal{A}(S)$. Then for any $c\in \mathcal{P}\in P(x_0,x_m)$ there is $0\leq i<m$ with $c^*\in \mathcal{P}^*\in P(x_i,x_{i+1})$ at distance $\leq k$ from $c$.
\end{lemma}

\begin{proof}[Proof of Proposition~\ref{prop:close}]
Let $c\in \mathcal{P}\in P(a,b)$ be at maximal distance $k$ from $\mathcal{G}$. Assume $k\geq 1$. Consider the maximal subpath $a'b'\subset \mathcal{P}$ containing $c$ with $a',b'$ at distance $\leq 2k$ from $c$. By Lemma~\ref{lem:subpath} we have $a'b'\in P(a,b)$. Let $a'',b''\in \mathcal{G}$ be closest to $a',b'$. Thus $|a'',a'|\leq k, |b'',b'|\leq k$, and in the case where $a'=a$ or $b'=b$, we have $a''=a$ or $b''=b$ as well. Hence $|a'',b''|\leq 6k$. Consider the concatenation of $a''b''$ with any geodesic paths $a'a'', b''b'$. Denote the consecutive vertices of that concatenation by $x_0,\ldots, x_m$, where $m\leq 8k$. By Lemma~\ref{lem:log}, the vertex $c$ is at distance $\leq \lceil \log_2 8k\rceil$ from some $x_i$. If $x_i\notin \mathcal{G}$, say $x_i\in a'a''$ then $|c,x_i|\geq |c,a'|-|a',x_i|\geq k$, so that $\lceil \log_2 8k\rceil \geq k$. Otherwise if $x_i\in \mathcal{G}$, then we also have $\lceil \log_2 8k\rceil \geq k$, this time by the definition of $k$. This gives $k\leq 6$.
\end{proof}

\begin{proof}[Proof of Theorem~\ref{thm:arc}]
Let $abd$ be a triangle in $\mathcal{A}(S)$ formed by geodesic edge-paths. By Lemma~\ref{lem:minsize1}, there are pairwise adjacent vertices $c_{ab},c_{ad},c_{db}$ on some paths in $P(a,b), P(a,d),P(b,d)$. We now apply Proposition~\ref{prop:close} to $c_{ab},c_{ad},c_{db}$ finding vertices on $ab,ad,bd$ at distance $\leq 6$. Thus $abd$ is $7$--centred at $c_{ab}$.
\end{proof}

\section{Curve graphs are hyperbolic}
\label{sec:curve}

In this section let $|\cdot,\cdot|$ denote the combinatorial distance in $\mathcal{AC}(S)$ instead of in $\mathcal{A}(S)$.

\begin{rem}[{\cite[Lem 2.2]{MM2}}]
\label{rem:quasi}
Suppose that $\mathcal{C}(S)$ is connected and hence $S$ is not the four holed sphere or the once holed torus.
Consider a retraction $r\colon \mathcal{AC}^{(0)}(S)\rightarrow \mathcal{C}^{(0)}(S)$ assigning to each arc a boundary component of a regular neighbourhood of its union with $\partial S$. We claim that $r$ is $2$--Lipschitz. If $S$ is not the twice holed torus, the claim follows from the fact that a pair of disjoint arcs does not fill $S$. Otherwise, assume that $a,b$ are disjoint arcs filling the twice holed torus $S$. Then the endpoints of $a,b$ are all on the same component of $\partial S$ and $r(a), r(b)$ is a pair of curves intersecting once. Hence the complement of $r(a)$ and $r(b)$ is a twice holed disc, so that $r(a),r(b)$ are at distance $2$ in $\mathcal{C}(S)$ and the claim follows.

Moreover, if $b$ is a curve in $\mathcal{AC}^{(0)}(S)$ adjacent to an arc $a$, then $b$ is adjacent to $r(a)$ as well. Thus the
distance in $\mathcal{C}(S)$ between two nonadjacent vertices $c,c'$ does not exceed $2|c,c'|-2$.
Consequently, a geodesic in $\mathcal{C}(S)$ is a $2$--quasigeodesic in $\mathcal{AC}(S)$.
Here we say that an edge-path with vertices $(c_i)_i$ is a \emph{$2$--quasigeodesic}, if $|i-j|\leq 2|c_i,c_j|$.
\end{rem}

\begin{proof}[Proof of Theorem~\ref{thm:curve}]
We first assume that $S$ has nonempty boundary. Let $T=abd$ be a triangle in the curve graph formed by geodesic edge-paths. By Remark~\ref{rem:quasi}, the sides of $T$ are $2$--quasigeodesics in $\mathcal{AC}(S)$. Choose arcs $\bar{a},\bar{b},\bar{d}\in \mathcal{AC}^{(0)}(S)$ that are adjacent to $a,b,d$, respectively.

Let $k$ be the maximal distance from any vertex $\bar{c}\in\mathcal{P}\in P(\bar{a}\bar{b})$ to the side $\mathcal{G}=ab$. Assume $k\geq 1$. As in the proof of Proposition~\ref{prop:close}, consider the maximal subpath $a'b'\subset \mathcal{P}$ containing $\bar{c}$ with $a',b'$ at distance $\leq 2k$ from $\bar{c}$. Let $a'',b''\in \mathcal{G}$ be closest to $a',b'$, so that $|a'',b''|\leq 6k$. Consider the concatenation $(x_i)_{i=0}^m$ of $a''b''$ with any geodesic paths $a'a'', b''b'$ in $\mathcal{AC}(S)$. Since $a''b''$ is a $2$--quasigeodesic, we have $m\leq 2k+ 2|a'',b''|=14k$.
For $i=0,\ldots, m-1$ let $\bar{x}_i\in\mathcal{AC}^{(0)}(S)$ be an arc adjacent (or equal) to both $x_i$ and $x_{i+1}$. Note that then all paths in $P(\bar{x}_i,\bar{x}_{i+1})$ are at distance $1$ from $x_{i+1}$.
By Lemmas~\ref{lem:subpath} and~\ref{lem:log}, the vertex $\bar{c}$
at distance $\leq \lceil \log_2 14k\rceil$ from a path in some $P(\bar{x}_i,\bar{x}_{i+1})$. Hence $\lceil \log_2 14k\rceil +1\geq k$.
This gives $k\leq 8$.

By Lemma~\ref{lem:minsize1}, there are pairwise adjacent vertices on some paths in $P(\bar{a},\bar{b}), P(\bar{a},\bar{d}),$ and in $P(\bar{b},\bar{d})$.
Let $\bar{c}$ be one of these vertices. Then $\bar{c}$ is at distance $\leq 9$ from all the sides of $T$ in $\mathcal{AC}(S)$. Consider the curve $c=r(\bar{c})$ adjacent to $\bar{c}$, where $r$ is the retraction from Remark~\ref{rem:quasi}. Then $T$ considered as a triangle in $\mathcal{C}(S)$ is $17$--centred at $c$, by Remark~\ref{rem:quasi}. Hence $\mathcal{C}(S)$ is $17$--hyperbolic for $\partial S\neq\emptyset$.

The curve graph $\mathcal{C}(S)$ of a closed surface (if connected) is known to be a $1$--Lipschitz retract of the curve graph $\mathcal{C}(S')$, where $S'$ is the once punctured $S$ \cite[Lem 3.6]{Har}, \cite[Thm 1.2]{RS}. The retraction is the puncture forgetting map. A section $\mathcal{C}(S)\rightarrow \mathcal{C}(S')$ can be constructed by choosing a hyperbolic metric on $S$, realising curves as geodesics and then adding a puncture outside the union of the curves. Hence $\mathcal{C}(S)$ is $17$--hyperbolic as well.
\end{proof}


\end{document}